\documentclass[a4paper,11pt]{amsart}
\usepackage{amsmath,amssymb,colordvi}
\usepackage{verbatim}
\newtheorem{theorem}{Theorem}
\newtheorem{corollary}[theorem]{Corollary}
\newtheorem{lemma}[theorem]{Lemma}
\newtheorem{example}[theorem]{\it Example}
\newtheorem{proposition}[theorem]{Proposition}

\newtheorem{remark}[theorem]{\it Remark}

\usepackage{amsthm}

\usepackage{mathrsfs}

\usepackage{amsmath}

\usepackage{color}

\usepackage{amssymb}

\def\NN{\mathbb N}
\def\N{\mathbb N}

\def\R{\mathbb R}

\def\RR{\mathbb R}

\begin{document}

\title[Quadrature rules for $L^1$-weighted norms of orthogonal polynomials]
{Quadrature rules for $L^1$-weighted norms of orthogonal polynomials}

\author{Luciano Abadias}
\address{Departamento de Matem\'aticas, Instituto Universitario de Matem\'aticas y Aplicaciones, Universidad de Zaragoza, 50009 Zaragoza, Spain.}
\email{labadias@unizar.es}

\author{Pedro J. Miana}
\address{Departamento de Matem\'aticas, Instituto Universitario de Matem\'aticas y Aplicaciones, Universidad de Zaragoza, 50009 Zaragoza, Spain.}
\email{pjmiana@unizar.es}

\author{Natalia Romero}
\address{Departamento de Matem\'aticas y Computaci\'{o}n, Universidad de La Rioja, 26006 Logro\~{n}o, Spain}
\email{natalia.romero@unirioja.es}

\thanks{\noindent Two first authors have been partially supported by Project MTM2010-16679, DGI-FEDER, of the MCYTS and Project E-64, D.G. Arag\'on, Spain. The third author is supported by Ministerio de Ciencia y Tecnolog\'{\i}a MTM2011-28636-C02.}




\begin{abstract}In this paper we obtain $L^1$-weighted norms of classical orthogonal polynomials (Hermite, Laguerre and Jacobi polynomials) in terms of the zeros of these orthogonal polynomials; these expressions are usually known as quadrature rules. In particular these new formulae   are useful to calculate directly some positive defined integrals as several examples show.

\end{abstract}

\date{}

\maketitle

\section{Introduction}

A unified approach to classical  orthogonal polynomials (Laguerre, Hermite and Jacobi polynomials)  is via Rodrigues' formula, i.e., \begin{equation}\label{rodrigues}Q_{n,\omega}(t)=\frac{1}{\omega(t)\mu_n}\frac{d^n}{dt^n}(\omega Q^n)(t),\qquad t \in (a,b),\end{equation} where $\omega$ is a weight in the range of definition $(a,b),$  $Q$ is a polynomial and $\mu_n$ is a constant depending of $n\ge 0$. The following table shows how to obtain  Laguerre, Hermite and Jacobi polynomials ($L_n^{(\alpha)}$, $H_n$, and $P_n^{(\alpha, \beta)}$ respectively) taking different values of $\omega$, $Q$ and $\mu_n$:\\

\medskip
\centerline{
\begin{tabular}{|c|c|c|c|c|}
\hline
Orthogonal polynomial, $Q_{n,\omega}$ & $\mu_n$ & $\omega$ & $Q$ & $(a,b)$ \\
\hline
Laguerre polynomial, $L_n^{(\alpha)}$ & $n!$ & $t^{\alpha}e^{-t}$ & $t$ & $(0,\infty)$ \\
\hline
Hermite polynomial, $H_n$ & $(-1)^n$ & $e^{-t^2}$ & $1$ & $(-\infty,\infty)$ \\
\hline
Jacobi polynomial, $P_n^{(\alpha,\beta)}$ & $(-1)^n 2^n n!$ & $(1-t)^{\alpha}(1+t)^{\beta}$ & $(1-t^2)$ & $(-1,1)$ \\
\hline
\end{tabular}
}
\medskip
 The zeros of the orthogonal polynomials $Q_{n,\omega}$ associated to the distribution $\omega(t)dt$ on the interval $[a,b]$ are real and distinct and are located in the interior of the interval $[a, b]$, see \cite[Theorem 3.3.1]{Szego}.  Note that $\mathcal{Z}(Q_{n,\omega})=\mathcal{Z}(wQ_{n,\omega})$ on the interval $(a, b)$ where  $\mathcal{Z}(f)$  is the set the zeros of the function $f$.

The well-known  Gauss-Jacobi quadrature rule  states that
$$
\int_a^bp(t)\omega(t)dt=\sum_{j=1}^n \lambda_jp(t_j), \qquad t_j\in \mathcal{Z}(Q_{n,\omega}),
$$
where $p$ is an arbitrary polynomial of degree $2n-1$ and parameters $(\lambda_j)_{1\le j\le n}$ are known as Christoffel numbers. The distribution $\omega(t)dt$ and the integer $n$ uniquely determine these numbers $(\lambda_j)_j$, see for example \cite[Theorem 3.4.1; Chapter XV]{Szego}. It is difficult to state the origin of this theorem but Jacobi must have been aware of it in 1826 (\cite{Jacobi}).

There exists a great number of papers and monographies about location of zeros of orthogonal polynomials and different types of quadrature rules: details of Gauss-Jacobi quadrature rule may be found, for example, in  \cite{GM}, \cite{SZ} and \cite[Chapter XV]{Szego}. Szeg\"{o} polynomials and Szeg\"{o} quadrature formula on the unit circle are studied for the Fej\'{e}r kernel in \cite{Santos}. Connections with
orthogonal polynomials on the line and Pad\'{e} approximants are also obtained in \cite{Santos}. In \cite{Hunter}, a number of  formulae  are derived for the numerical evaluation of singular integrals in the interval $(-1, 1)$. These formulae are based on Gauss-Legendre quadrature rule.
 Later in \cite{BVM}, authors propose to approximate the Hilbert transform of smooth functions   by using the zeros of Hermite polynomials. In the nice paper \cite{Gaut}, various concepts of orthogonality on the real line  are reviewed in connection with quadrature rules. Finally, Gaussian and other positive quadrature rules are investigated to deduce some conditions about the existence of prescribed abscissa in \cite{BBMQ}.

In this paper, we prove a formula similar to   Gauss-Jacobi quadrature rule (also named as Gaussian quadrature rule) to obtain $L^1$-weighted  norm of classical orthogonal polynomials. This kind of results  seems that  has not been considered before in the literature. In \cite{CS, DS}, the error of the Gaussian quadrature rule is estimated in an $L^1$-weighted norm, only in the Jacobi setting.

A first approach to our problem is the following theorem. Again, the set of zeros of orthogonal polynomials plays an important role.

\begin{theorem}\label{join} For $n\geq 1,$ functions $Q_{n,\omega}$ verify   $$\Vert  Q_{n,\omega}\,\omega\Vert_1:= \int_{a}^b \vert Q_{n,\omega}(t)\vert \omega(t)dt=2\frac{\mu_{n-1}}{|\mu_n|}\displaystyle\sum_{j=1}^{n}(-1)^{j+1}\omega(t_j)Q(t_j)Q_{n-1,\omega Q}(t_j),$$ where $t_j\in\mathcal{Z}(Q_{n,\omega})$ and $a<t_1<\ldots<t_n<b.$
\end{theorem}
\begin{proof} We call $t_0=a$ and $t_{n+1}=b.$ From Rodrigues' formula,  we obtain that
\begin{eqnarray*}
\int_{a}^b \vert Q_{n,\omega}(t)\vert \omega(t) dt&=&\displaystyle\sum_{j=0}^{n}\int_{t_j}^{t_{j+1}}|\frac{1}{\mu_n}\frac{d^n}{dt^n}(\omega Q^n)(t)|\,dt=\displaystyle\sum_{j=0}^{n}{(-1)^{j}\over |\mu_n|}\int_{t_j}^{t_{j+1}}\frac{d^n}{dt^n}( \omega Q^n)(t)\,dt \\
&=&\displaystyle\sum_{j=0}^{n}{(-1)^{j}\over {|\mu_n|}}\frac{d^{n-1}}{dt^{n-1}}( \omega Q^n)(t)\Big|_{t_j}^{t_{j+1}}\,=
2\sum_{j=1}^{n}{(-1)^{j+1}\over |\mu_n|}\frac{d^{n-1}}{dt^{n-1}}(\omega Q^n)(t_j)\,
\end{eqnarray*}
 where we have used that $\displaystyle{|{1\over \mu_n}\frac{d^n}{dt^n}(\omega Q^n)|(t)=(-1)^{j}{1\over \vert\mu_n\vert}\frac{d^n}{dt^n}(\omega Q^n)(t)}$ for $t_j<t<t_{j+1},$ and the function $\omega Q^n$ and its derivatives of order less than $n$ vanish at the endpoints $a$ and $b$ in the three cases (Laguerre, Hermite, and Jacobi polynomials considered in the Introduction). Now we apply the formula (\ref{rodrigues}) to get that
 $$
\int_{a}^b \vert Q_{n,\omega}(t)\vert \omega(t) dt=2\frac{\mu_{n-1}}{|\mu_n|}\displaystyle\sum_{j=1}^{n}(-1)^{j+1}\omega(t_j)Q(t_j)Q_{n-1,\omega Q}(t_j),
$$
and we conclude the result.
\end{proof}
In fact, this result may be improved using some recurrence relations; we  show that
$$
\int_{a}^b \vert Q_{n,\omega}(t)\vert \omega(t) dt=2c_{n, \omega}\displaystyle\sum_{j=1}^{n}(-1)^{j+1}\omega(t_j)Q_{n-1,\omega }(t_j),
$$
 in Corollaries \ref{aaa}, \ref{ruba} and \ref{calva} (where $c_{n, \omega}$ is a parameter which depends on $\omega$ and $n$).

In this paper we are interested to estimate and calculate the following $L^1$-weighted norms
\begin{equation}\label{general} \Vert {t^i\over i!}\ Q_{n,\omega} \omega\Vert_1=
\int_a^{b}{\vert t\vert^i\over i!}\vert Q_{n,\omega}(t)\vert \omega(t)dt, \qquad n, i \in \N\cup\{0\},
\end{equation}
 in the  setting of classical orthogonal polynomials $Q_{n,\omega}$. These $L^1$-weighted norms are commonly used in applied and mathematical analysis and related to Sobolev norms (see Remark \ref{sobo}). Although a unified presentation might be considered (see Theorem \ref{join} and compare Lemmata \ref{laguerre}, \ref{keyhermite} and \ref{jacobi}), we dedicate  different sections to results concerning about each  family. The  aim of this point of view is twofold: firstly, the situation of the number $0$ respect to the set $\mathcal{Z}(Q_{n,\omega})$ is  essential and different in each case; and secondly, it allows  to handle easily constants and parameters  involved in every case.

The main line of reasoning is to study a family of    functions defined by
$
q_{n,\omega}:=\displaystyle{1\over k_n}\omega Q_{n,\omega}$
where the constant $k_n$ is given by the orthogonal condition,

\begin{equation}\label{ortos}
\int_a^b   Q_{n,\omega}(t) Q_{m,\omega}(t)\omega(t)dt={k_n}\delta_{n, m},\qquad n,m\in \NN\cup\{0\},
\end{equation}
and $\delta_{n,m}$ is the Kronecker distribution. The exact value of $k_n$ in each case is presented in the next table:
$$
\centerline{
\begin{tabular}{|c|c|}
\hline
Orthogonal polynomial, $Q_{n,\omega}$ & $k_n$ \\
\hline
Laguerre polynomial, $L_n^{(\alpha)}$ & ${\frac{\Gamma(n+\alpha+1)}{n!}}$  \\
\hline
Hermite polynomial, $H_n$ & $2^n n!\sqrt{\pi}$  \\
\hline
Jacobi polynomial, $P_n^{(\alpha,\beta)}$ & $\frac{2^{\alpha+\beta+1}\Gamma(n+\alpha+1)\Gamma(n+\beta+1)}{(2n+\alpha+\beta+1)\Gamma(n+\alpha+\beta+1)n!} $  \\
\hline
\end{tabular}
}
$$

\medskip

 These functions $q_{n,\omega}$ are fundamental in classical orthogonal expansions (\cite[Chapter 4]{Lebedev} and \cite[Chapter IX]{Szego}). Recently the authors have treated them to introduce Laguerre expansions for $C_0$-semigroups in \cite{AM1} and Hermite expansions for $C_0$-groups and cosine function  in \cite{AM2}. In fact,  to get  sharp estimations of $\Vert q_{n,\omega}\Vert_1$ is the motivating starting-point of this paper:  sharp estimations allow to  assure convergence of  vector-valued  orthogonal expansions, see more details in \cite{AM1, AM2}.

The paper is organized as follows. The second section deal with Laguerre polynomials, the third section with Hermite polynomials and the last one with Jacobi polynomials.   Three recurrence relations, differential equations and other known relations for orthogonal polynomials are verified by  functions $
q_{n,\omega}$. We apply the Cauchy-Schwarz inequality to estimate (\ref{general}) in Propositions \ref{si}, \ref{bound} and \ref{yes}. Then we integrate by parts to express
$\displaystyle{
\int{ t^i\over i!} q_{n,\omega}(t)dt}
$
by linear combinations of functions $q_{k,\omega_k}$ (Lemmata \ref{laguerre}, \ref{keyhermite} and \ref{jacobi}). A straightforward consequence of this identity is the equality
$$
\int_a^b{ t^i\over i!} q_{n,\omega}(t)dt=0, \qquad 0\le i \le n-1,
$$
which may be also shown from the orthogonal relation (\ref{ortos}).

As consequences of previous results, Theorems \ref{aa}, \ref{zeroshermite} and \ref{jacobino} are  main results of this paper, where the exact value of (\ref{general})  is obtained for $0\le i \le n-1$. For  $i=0$ or $i=n-1$, these theorems are improved in Corollaries \ref{aaa}, \ref{ruba} and \ref{calva}.
  These formulae provide a fast and efficient way to calculate some defined integrals, as Examples \ref{ejlaguerre}, \ref{ejhermite} and \ref{jaejemplo} show, and may be of interest to  general  and specific public including mathematical software companies.


\section{Laguerre polynomials}

\setcounter{theorem}{0}
\setcounter{equation}{0}

Generalized Laguerre polynomials $\{ L_n^{(\alpha)} \}_{n\ge 0}$ ($\alpha >-1$) are given  by
$$
L_n^{(\alpha)}(t)=\sum_{k=0}^n(-1)^{k}{n+\alpha\choose n-k}{t^k\over k!}, \qquad t\ge 0;
$$
in particular $L_0^{(\alpha)}(t)=1$, $L_1^{(\alpha)}(t)=-t+\alpha+1$  and
$\displaystyle{L_2^{(\alpha)}(t)={t^2\over 2}-(\alpha+2)t+{(\alpha+2)(\alpha+1)\over 2}
}$. Polynomials $\{ L_n^{(\alpha)} \}_{n\ge 0}$ are solutions of second order differential equation \begin{equation}\label{Laguerresoleq}ty''+(\alpha+1-t)y'+ny=0,\end{equation}
and  satisfy the following recurrence relations
\begin{eqnarray*}
nL_n^{(\alpha)}(t)&=&(n+\alpha)L_{n-1}^{(\alpha)}(t)-tL_{n-1}^{(\alpha+1)}(t);\cr \cr
tL_n^{(\alpha+1)}(t)&=&(n+\alpha)L_{n-1}^{(\alpha)}(t)-(n-t)L_{n}^{(\alpha)}(t),
\end{eqnarray*}
see for example \cite{Lebedev} and \cite{Szego}. Note that  $L_n^{\alpha}(t)=L_n^{(\alpha+1)}(t)-L_{n-1}^{(\alpha+1)}(t)$ and we iterate to get that
\begin{equation}\label{euro}
L_n^{(\alpha)}(t)=\sum_{k=0}^n(-1)^kL^{(\alpha+1+k)}_{n-k}(t), \qquad t\ge 0.
\end{equation}

Now we consider the following Laguerre functions $\{ \ell_n^{(\alpha)} \}_{n\ge 0}$ defined by
$$
\ell_n^{(\alpha)}(t):=\frac{n!}{\Gamma(n+\alpha+1)}t^{\alpha}e^{-t}L_n^{(\alpha)}(t),\qquad t\ge 0,
$$
for $\alpha\neq -1,-2,-3,\ldots,$  and $n\in \NN\cup\{0\}$. Recently these functions have been studied in \cite{AM1}, and  the following identity \begin{equation}\label{Laguerre2}\ell_n^{(\alpha)}(t)=\ell_{n-1}^{(\alpha)}(t)-\ell_{n-1}^{(\alpha+1)}(t),\qquad t\ge 0,\end{equation} holds, see \cite[Proposition 2.3 (i)]{AM1}.
The family $\{ \ell_n^{(\alpha)} \}_{n\geq 0}$ is a total set in $L^p(\R^+)$ for $\alpha >-\frac{1}{p},$ with $1\leq p<\infty$ (\cite[Theorem 3.1 (ii)]{AM1}). Furthermore, the optimal estimate of $\lVert\ell_n^{(\alpha)}\rVert_1$ has a key role on the study of vector-valued Laguerre expansions.  In \cite[Remark 2.10]{AM1}, we prove that  $$\frac{M_\alpha}{n^{\frac{\alpha+1}{2}}}\leq \Vert \ell^{(\alpha)}_n\Vert_1 \le \frac{C_{\alpha}}{n^{\alpha \over 2}}, \qquad n\in \NN,$$  for $\alpha>-1$  and $C_{\alpha}, M_{\alpha}>0.$

\begin{proposition} \label{si}For $n, i\in \N$, and  $\alpha>-(i+1)$, the inequality
$$
\int_0^\infty {t^i\over i!}\vert \ell_n^{(\alpha)}(t)\vert dt\le C_\alpha{2^ii^{{\alpha\over 2}-{1\over 4}}\over n^{\alpha\over 2}},
$$ holds with $C_\alpha$  a  constant which does not depend on $n$ or $i$.

\end{proposition}

\begin{proof} We apply the Cauchy-Schwarz inequality to get that
\begin{eqnarray*}
\int_0^\infty {t^i\over i!}\vert \ell_n^{(\alpha)}(t)\vert dt&=&{n!\over i!\,\Gamma(\alpha+n+1)}\int_0^\infty t^{i+\alpha}e^{-t}\vert L_n^{(\alpha)}(t)\vert dt\cr
&\le&{n!\over i!\,\Gamma(\alpha+n+1)}\left(\int_0^\infty t^{2i+\alpha}e^{-t}dt\right)^{1\over 2}\left(\int_0^\infty t^{\alpha}e^{-t}\vert L_{n}^{(\alpha)}(t)\vert^2dt\right)^{1\over 2}\cr
&=&\left({n!\,\Gamma(2i+\alpha+1)\over (i!)^2\,\Gamma(\alpha+n+1)}\right)^{1\over 2},
\end{eqnarray*}
where we have used that functions $t\mapsto \left({n!\over \Gamma(\alpha+n+1)}\right)^{1\over 2}t^{\alpha\over 2} e^{-t\over 2}L_n^{(\alpha)}(t)$ form a Hilbertian basis on $L^2(\R^+)$.

Since $\displaystyle{\lim_{n\to \infty}\displaystyle{\Gamma(n+\alpha)\over \Gamma(n)n^\alpha}=1}$, we deduce that
$$
{n!\,\Gamma(2i+\alpha+1)\over (i!)^2\,\Gamma(\alpha+n+1)}\le C_\alpha{(2i)!(2i+1)^\alpha\over n^\alpha(i!)^2}\le C_\alpha{2^{2i}i^{\alpha-{1\over 2}}\over n^\alpha},
$$
where we have applied
Stirling's formula and $C_\alpha$ is a constant depend on $\alpha$ and  independent on $i$ and $n$. We conclude the result.
\end{proof}

\begin{lemma}\label{laguerre} For $n\in \N$ and $0\le i\le n-1$, the Laguerre functions $\ell_n^{(\alpha)}$ satisfy
$$
\int  {t^i\over i!} \ell_n^{(\alpha)}(t) dt= \sum_{k=0}^i{(-1)^k\over (i-k)!}t^{i-k} \ell_{n-1-k}^{(\alpha+1+k)}(t).
$$
\end{lemma}
\begin{proof} We integrate by parts $i$-times to get that
\begin{eqnarray*}
\int  {t^i\over i!} \ell_n^{(\alpha)}(t) dt&=&{1\over\Gamma(\alpha+n+1)}\int {t^i\over i!}\frac{d^{n}}{dt^{n}}(t^{n+\alpha}e^{-t})(t)\,dt\cr
&=&{1\over\Gamma(\alpha+n+1)}\displaystyle\sum_{k=0}^{i}\frac{(-1)^k }{(i-k)!}t^{i-k}\frac{d^{n-k-1}}{dt^{n-k-1}}(t^{n+\alpha}e^{-t})(t)\\
&=&\displaystyle\sum_{k=0}^{i}\frac{(-1)^k }{(i-k)!}t^{i-k}\ell_{n-1-k}^{(\alpha+1+k)}(t),
\end{eqnarray*}
and we conclude the result.
\end{proof}

\begin{theorem} \label{aa} Let $n\in\N\cup\{0\},$  $0\leq i\leq n-1$ and $\alpha>-1$. Then the Laguerre functions $\ell_n^{(\alpha)}$ verify
\begin{equation}\label{lag7} \int_0^\infty {t^i\over i!}\vert \ell^{(\alpha)}_n(t)\vert dt = 2\sum_{m=1}^n(-1)^{m+1}\sum_{k=0}^i{(-1)^{k}\over (i-k)!}t_m^{i-k} \ell_{n-1-k}^{(\alpha+1+k)}(t_m),
\end{equation} with $t_m\in\mathcal{Z}(L_{n}^{(\alpha)})=\{t_1<\ldots<t_{n}\}.$
\end{theorem}

\begin{proof}  We write by $t_0=0$ and $t_{n+1}=+\infty.$ Note that $L_n^{(\alpha)}(0)=\binom{n+\alpha}{n}>0$, $|L_n^{(\alpha)}(t)|=(-1)^{m}L_n^{(\alpha)}(t)$ for $t_m<t<t_{m+1}$ and then
$$
 \int_0^\infty {t^i\over i!}\vert \ell^{(\alpha)}_n(t)\vert dt=\frac{1}{i!}\sum_{m=0}^n(-1)^{m}\int_{t_m}^{t_{m+1}}t^{i}\ell_{n}^{(\alpha)}(t)\,dt.
$$ We apply Lemma \ref{laguerre} to deduce that
 \begin{eqnarray*}
\int_0^\infty {t^i\over i!}\vert \ell^{(\alpha)}_n(t)\vert dt&=& \sum_{m=0}^n(-1)^{m}\sum_{k=0}^i{(-1)^k\over (i-k)!}t^{i-k} \ell_{n-1-k}^{(\alpha+1+k)}(t)\Big|_{t_m}^{t_{m+1}}\cr
&=& 2\sum_{m=1}^n(-1)^{m+1}\sum_{k=0}^i{(-1)^{k}\over (i-k)!}t_m^{i-k} \ell_{n-1-k}^{(\alpha+1+k)}(t_m),
\end{eqnarray*}
and we have used that  $\displaystyle\lim_{t\to 0+}t^{i-k}\ell_{n-1-k}^{(\alpha+1+k)}(t)=0=\displaystyle\lim_{t\to \infty}t^{i-k}\ell_{n-1-k}^{(\alpha+1+k)}(t).$
\end{proof}

\begin{corollary}\label{aaa} For $\alpha>-1$ and $n\in\N,$ the Laguerre functions $\ell_n^{(\alpha)}$ verify
\begin{eqnarray*} \lVert\ell_n^{(\alpha)}\rVert_1 &=&2\displaystyle\sum_{m=1}^n(-1)^{m+1}\ell_{n-1}^
{(\alpha)}(t_m),\\
 \int_0^\infty {t^{n-1}}\vert \ell^{(\alpha)}_n(t)\vert dt &=& {2 \over (\alpha+n)}\sum_{m=1}^n(-1)^{m+1}t_m^{n} \ell_{n-1}^{(\alpha)}(t_m),
\end{eqnarray*}
with $t_m\in\mathcal{Z}(L_n^{(\alpha)})=\{t_1<\ldots<t_n\}.$
\end{corollary}

\begin{proof}  To obtain the first equality, take $i=0$ in the equation (\ref{lag7}) and we use that $\ell_{n-1}^{(\alpha+1)}(t_m)=\ell_{n-1}^{(\alpha)}(t_m)$ for $ t_m\in\mathcal{Z}(L_n^{(\alpha)})$ by equality \eqref{Laguerre2}. Taking $i=n-1$ in (\ref{lag7}), we get that
\begin{eqnarray*}
\int_0^\infty {t^{n-1}\over (n-1)!}\vert \ell^{(\alpha)}_n(t)\vert dt &=&
2\sum_{m=1}^n(-1)^{m+1}\sum_{k=0}^{n-1}{(-1)^{k}\over (n-1-k)!}t_m^{n-1-k} \ell_{n-1-k}^{(\alpha+1+k)}(t_m)\cr
&=&{2 \over \Gamma(\alpha+n+1)}\sum_{m=1}^n(-1)^{m+1}t_m^{\alpha+n}e^{-t_m} \sum_{k=0}^{n-1}{(-1)^{k}} L_{n-1-k}^{(\alpha+1+k)}(t_m)\cr
&=&{2 \over (n-1)!(\alpha+n)}\sum_{m=1}^n(-1)^{m+1}t_m^{n} \ell_{n-1}^{(\alpha)}(t_m),
\end{eqnarray*}
where we have applied the equality (\ref{euro}).
\end{proof}

\begin{remark}\label{sobo}{\rm  Note that the first equality in Corollary   \ref{aaa}  improves the inequality
$$\lVert \ell_n^{(\alpha)} \rVert_1\geq \displaystyle\max_{t\in\mathcal{Z}(L_n^{(\alpha)})}|\ell_{n-1}^{(\alpha)}(t)|, \qquad n\geq 1,$$
 shown in \cite[Theorem 2.4 (iv)]{AM1}. In other hand, the equality
 $$\frac{d^k}{dt^k}\ell_{n}^{(\alpha)}(t)=\ell_{n+k}^{(\alpha-k)}(t), \qquad t\ge 0,$$
holds for   $k\geq 1$ (\cite[Proposition 2.3 (vi)]{AM1}) and then,   we  obtain  the following Sobolev norms
$$
\int_0^\infty \vert \frac{d^k}{dt^k}\ell_{n}^{(\alpha)}(t)\vert dt=\int_0^\infty \vert \ell_{n+k}^{(\alpha-k)}(t)\vert dt= 2\displaystyle\sum_{m=1}^{n+k}(-1)^{m+1}\ell_{n+k-1}^
{(\alpha-k)}(t_m),
$$
for $\alpha>k-1$ and $t_m\in\mathcal{Z}(L_{n+k}^{(\alpha-k)})$.

}

\end{remark}
\begin{example} \label{ejlaguerre} {\rm We consider $\ell_2^{(0)}(t)={1\over 2}e^{-t}(t^2-4t+2)$. By Corollary \ref{aaa}  we conclude that
\begin{eqnarray*}
\int_0^\infty  {1\over 2}e^{-t}\vert t^2-4t+2\vert dt&=&2e^{-2}\left(e^{\sqrt{2}}(\sqrt{2}-1)+e^{-\sqrt{2}}(1+\sqrt{2})\right); \cr
\int_0^\infty  {1\over 2}e^{-t}t\vert t^2-4t+2\vert dt&=& 2e^{-2}\left(e^{\sqrt{2}}(5\sqrt{2}-7)+e^{-\sqrt{2}}(5\sqrt{2}+7)\right).
\end{eqnarray*}
Now we take $\ell^{(2)}_1(t)={1\over 6} t^2( 3-t) e^{-t}$ to check that
$\displaystyle{
{1\over 6}\int_0^\infty t^2\vert 3-t\vert e^{-t}dt=9e^{-3}.}
$
Finally we take
$
\ell_2^{(1)}(t)={1\over 6}t( t^2-6t+6 ) e^{-t}
$
to get that
\begin{eqnarray*}
{1\over 6}\int_0^\infty t\vert t^2-6t+6 \vert e^{-t}dt&=&{e^{-3+\sqrt{3}}}(4\sqrt{3}-6)+{e^{-3-\sqrt{3}}}(4\sqrt{3}+6),\cr
{1\over 6}\int_0^\infty t^2\vert t^2-6t+6 \vert e^{-t}dt&=&2 e^{-3}\left(e^{\sqrt{3}}\left(14\sqrt{3}-24\right)+
e^{-\sqrt{3}}\left(14\sqrt{3}+24\right)\right).
\end{eqnarray*}
}

\end{example}

\section{Hermite polynomials}

\setcounter{theorem}{0}
\setcounter{equation}{0}

Hermite polynomials are polynomial solutions of second order differential equation \begin{equation}\label{ecdif}y''-2ty'+2ny=0. \end{equation} First Hermite polynomials are the following ones:
$$
H_0(t)=1; \qquad H_1(t)=2t;\qquad H_2(t)=4t^2-2;\qquad H_3(t)=4t(2t^2-3).
$$

In the following we consider a family of Hermite functions in $\R$ defined by
$$
h_n(t):=\frac{1}{2^n n!\sqrt{\pi}}e^{-t^2}H_n(t),\qquad t\in \RR,
$$
for $n\in \NN\cup\{0\}$. They have been studied in detail in \cite[Section 2]{AM2} and, in particular,  the following identity  \begin{equation}\label{Hermite}h_{n}^{(k)}=(-1)^k 2^k(n+1)\ldots (n+k)h_{n+k},\end{equation} is proved in  \cite[Proposition 2.3 (iii)]{AM2}. The family $\{ h_n \}_{n\geq 0}$ is a total set in $L^p(\R),$ with $1\leq p<\infty,$ and the optimal estimate of $\lVert h_n\rVert_1$ has a great importance on the study of vector-valued Hermite expansions, see more details in \cite{AM2}. By standard techniques,  the known bound   $$\Vert h_n\Vert_1 \leq \frac{1}{\sqrt{n! 2^n}}, \qquad n\in \NN,$$ is shown, see for example \cite[Remark 2.5]{AM2}. In the next proposition, we consider $L^1$-weighted norms.

\begin{proposition}\label{bound}For $n, i \in \N$, the Hermite functions $h_n$ verify
$$
\int_{-\infty}^{\infty}{\vert t\vert^i\over i!}\vert h_n(t)\vert dt\le {1\over \sqrt{2^n\,n! \,i!\sqrt{\pi}}}.
$$
\end{proposition}

\begin{proof} We apply the Cauchy-Schwarz  inequality to obtain that
\begin{eqnarray*}
\int_{-\infty}^{\infty}{\vert t\vert^i\over i!}\vert h_n(t)\vert dt
&\le&{2\over i!\sqrt{2^n n!\sqrt{\pi}}}\left(\int_0^\infty t^{2i}e^{-t^2}dt
\right)^{1\over 2}\left({1\over 2}\int_{-\infty}^{\infty}{e^{-t^2}\vert H_n(t)\vert^2\over 2^n n!\sqrt{\pi}} dt\right)^{1\over 2}\cr
&=&{2\over i!\sqrt{2^n n!\sqrt{\pi}}}\left({\Gamma(i+{1\over 2})\over 2}\right)^{1\over 2}\left({1\over 2}\right)^{1\over 2}\le {1\over \sqrt{2^n n! \, i!\,\sqrt{\pi}}},
\end{eqnarray*}
where we have used that functions $t\mapsto \displaystyle{e^{-t^2\over 2}H_n(t)\over \sqrt{2^n n!\sqrt{\pi}}}$ (for $n\ge 0$) form a Hilbertian basis on $L^2(\R)$.
\end{proof}


The proof of the next lemma runs parallel to the proof of Lemma \ref{laguerre} and we do not  include it here.

\begin{lemma} \label{keyhermite} Take $n\in \N$ and $0\le i\le n-1$. Then the following identity holds:
\begin{eqnarray*}\int {t^i\over i!}h_n(t)\,dt&=&-{1\over n!}\displaystyle\sum_{k=0}^{i}\frac{(n-1-k)!}{2^{k+1}(i-k)!}t^{i-k}h_{n-1-k}(t).
\end{eqnarray*}

\end{lemma}

\begin{theorem}\label{zeroshermite}  Let be $n\in \N$, $0\le i\le n-1$ and  $\mathcal{Z}(H_n)=\{ t_1<\ldots<t_{n}\}$.

\begin{itemize}
\item[(i)] If $i$ is even, then the Hermite functions $h_n$ satisfy
 \begin{eqnarray*} \int_{-\infty}^\infty{t^i\over i!} \vert h_n(t)\vert dt&=&{1\over n!}\displaystyle\sum_{m=1}^{n}(-1)^{m+n}\displaystyle\displaystyle\sum_{k=0}^{i}\frac{ (n-1-k)!}{2^{k}(i-k)!}t^{i-k}_mh_{n-1-k}(t_m).
\end{eqnarray*}
\item[(ii)]  If  $i$ is odd and $n$ even, then they verify
 \begin{eqnarray*} \int_{-\infty}^\infty{\vert t\vert^i\over i!} \vert h_n(t)\vert dt&=&{1\over n!}\displaystyle\sum_{m=1}^{{n\over 2}}(-1)^{m+1}\displaystyle\displaystyle\sum_{k=0}^{i}\frac{ (n-1-k)!}{2^{k}(i-k)!}t^{i-k}_mh_{n-1-k}(t_m)\cr
 &+&{(-1)^{{n\over 2}}(n-1-i)!\over n!2^i}h_{n-1-i}(0)\cr&+ &{1\over n!} \sum_{m={{n\over 2}+1}}^n(-1)^{m}\displaystyle\displaystyle\sum_{k=0}^{i}\frac{ (n-1-k)!}{2^{k}(i-k)!}t^{i-k}_mh_{n-1-k}(t_m);
\end{eqnarray*} and in the case that $n$ is odd,
 \begin{eqnarray*} \int_{-\infty}^\infty{\vert t\vert^i\over i!} \vert h_n(t)\vert dt&=&{1\over n!}\displaystyle\sum_{m=1}^{n-1\over 2}(-1)^{m}\displaystyle\displaystyle\sum_{k=0}^{i}\frac{ (n-1-k)!}{2^{k}(i-k)!}t^{i-k}_mh_{n-1-k}(t_m)\cr
 &\quad&+{1\over n!}\sum_{m={n+3\over 2}}^n(-1)^{m+1}\displaystyle\displaystyle\sum_{k=0}^{i}\frac{ (n-1-k)!}{2^{k}(i-k)!}t^{i-k}_mh_{n-1-k}(t_m).
\end{eqnarray*}
 \end{itemize}

\end{theorem}

\begin{proof} (i) We write $t_0=-\infty $ and $t_{n+1}=+\infty$ to get that
 \begin{eqnarray*}\int_{-\infty}^\infty{t^i\over i!} \vert h_n(t)\vert dt&=& \frac{1}{i!}\sum_{m=0}^n (-1)^{m+n} \int_{t_m}^{t_{m+1}}t^ih_n(t)dt\cr &=& {1\over n!}\sum_{m=0}^n (-1)^{m+n+1}\displaystyle\sum_{k=0}^{i}\frac{ (n-1-k)!}{2^{k+1}(i-k)!}t^{i-k}h_{n-1-k}(t)\Big|_{t_m}^{t_{m+1}}
 \cr &=& {1\over n!}\sum_{m=1}^n (-1)^{m+n}\displaystyle\sum_{k=0}^{i}\frac{ (n-1-k)!}{2^{k}(i-k)!}t^{i-k}_mh_{n-1-k}(t_m),
\end{eqnarray*}
where  we have applied the Lemma \ref{keyhermite}  and $\displaystyle\lim_{t\to \pm \infty}t^{i-k}h_{n-1-k}(t)=0.$

(ii) Since $n$ is even, then we prove that
\begin{eqnarray*}\int_{-\infty}^\infty{\vert t\vert^i\over i!} \vert h_n(t)\vert dt&=& \frac{1}{i!}\sum_{m=0}^{{n\over 2}-1} (-1)^{m+1} \int_{t_m}^{t_{m+1}}t^ih_n(t)dt+{(-1)^{{n\over 2}-1}\over i} \int_{t_{n\over 2}}^{0}t^ih_n(t)dt\cr &+& {(-1)^{{n\over 2}}\over i!} \int^{t_{{n\over 2}+1}}_{0}t^ih_n(t)dt+ \frac{1}{i!}\sum^{n}_{m={n\over 2}+1} (-1)^{m} \int_{t_m}^{t_{m+1}}t^ih_n(t)dt.\cr
\end{eqnarray*}
By Lemma \ref{keyhermite}, we deduce that
$$
\int^{t_{{n\over 2}+1}}_{0}t^ih_n(t)dt=-{1\over n!}\displaystyle\sum_{k=0}^{i}\frac{ i!(n-1-k)!}{2^{k+1}(i-k)!}t_{{n\over 2}+1}^{i-k}h_{n-1-k}(t_{{n\over 2}+1})+{i!(n-1-i)!\over 2^{i+1}n!}h_{n-1-i}(0),
$$
and then
\begin{eqnarray*}
&\quad& {(-1)^{{n\over 2}}\over i!} \int^{t_{{n\over 2}+1}}_{0}t^ih_n(t)dt+ \frac{1}{i!}\sum^{n}_{m={n\over 2}+1} (-1)^{m} \int_{t_m}^{t_{m+1}}t^ih_n(t)dt\cr
&\quad&=(-1)^{{n\over 2}}{(n-1-i)!\over 2^{i+1}n!}h_{n-1-i}(0)+ {1\over n!} \sum_{m={{n\over 2}+1}}^n(-1)^{m}\displaystyle\displaystyle\sum_{k=0}^{i}\frac{ (n-1-k)!}{2^{k}(i-k)!}t^{i-k}_mh_{n-1-k}(t_m).
\end{eqnarray*}
Analogously we get the identities for the first two summands and we conclude the result.

Finally we consider the case that $n$ is odd; in this case $t_{n+1\over 2}=0$ and we get that
\begin{eqnarray*}&\quad&\int_{-\infty}^\infty{\vert t\vert^i\over i!} \vert h_n(t)\vert dt= \frac{1}{i!}\sum_{m=0}^{{n-1\over 2}} (-1)^{m} \int_{t_m}^{t_{m+1}}t^ih_n(t)dt+ \frac{1}{i!}\sum^{n}_{m={n+1\over 2}} (-1)^{m+1} \int_{t_m}^{t_{m+1}}t^ih_n(t)dt\cr
&\quad&=\frac{1}{n!}\sum_{m=1}^{{n-1\over 2}} (-1)^{m}\displaystyle\sum_{k=0}^{i}\frac{ (n-1-k)!}{2^{k}(i-k)!}t^{i-k}_mh_{n-1-k}(t_m)+ (-1)^{{n+1\over 2}}{(n-1-i)!\over 2^{i+1}n!}h_{n-1-i}(0)\cr
&\quad&+(-1)^{{n-1\over 2}}{(n-1-i)!\over 2^{i+1}n!}h_{n-1-i}(0)+\frac{1}{n!}\sum^{n}_{m={n+3\over 2}} (-1)^{m+1}\displaystyle\sum_{k=0}^{i}\frac{ (n-1-k)!}{2^{k}(i-k)!}t^{i-k}_mh_{n-1-k}(t_m)\cr
&\quad&=\frac{1}{n!}\sum_{m=1}^{{n-1\over 2}} (-1)^{m}\displaystyle\sum_{k=0}^{i}\frac{ (n-1-k)!}{2^{k}(i-k)!}t^{i-k}_mh_{n-1-k}(t_m)\cr&\quad&\qquad\qquad+\frac{1}{n!}\sum^{n}_{m={n+3\over 2}} (-1)^{m+1}\displaystyle\sum_{k=0}^{i}\frac{ (n-1-k)!}{2^{k}(i-k)!}t^{i-k}_mh_{n-1-k}(t_m),
\end{eqnarray*}
and we conclude the result.
\end{proof}

\begin{corollary}\label{ruba} For $n\in\N$ the Hermite functions $h_n$ satisfy \begin{eqnarray*} \lVert h_n\rVert_1&=&\frac{1}{n}\displaystyle\sum_{m=1}^n(-1)^{m+n}h_{n-1}(t_m),\cr
\int_{-\infty}^\infty{t^{2n}\over (2n)!} \vert h_{2n+1}(t)\vert dt&=&{1\over (2n+1)!}\displaystyle\sum_{m=1}^{2n+1}(-1)^{m+1}\displaystyle\displaystyle\sum_{k=0}^{2n}\frac{ 1}{2^{k}}t^{2n-k}_mh_{2n-k}(t_m);\cr
 \int_{-\infty}^\infty{\vert t\vert^{2n-1}\over (2n-1)!} \vert h_{2n}(t)\vert dt&=&{1\over (2n)!}\displaystyle\sum_{m=1}^{{n}}(-1)^{m+1}\displaystyle\displaystyle\sum_{k=0}^{2n-1}\frac{ 1}{2^{k}}t^{2n-1-k}_mh_{2n-1-k}(t_m)+{(-1)^{{n}}\over (2n)!\,2^{2n-1}\sqrt{\pi}}\cr&+ &{1\over (2n)!} \sum_{m={{n}+1}}^{2n}(-1)^{m}\displaystyle\displaystyle\sum_{k=0}^{2n-1}\frac{ 1}{2^{k}}t^{2n-1-k}_mh_{2n-1-k}(t_m);
\end{eqnarray*} with $t_m\in\mathcal{Z}(H_n)=\{t_1<\ldots<t_n\}.$
\end{corollary}

\begin{example}\label{ejhermite}{\rm Let us consider functions $h_1,\, h_2 $ and $h_3$ defined by  $$\displaystyle{h_1(t)= {1\over \sqrt{\pi}}te^{-t^2}}, \qquad \displaystyle{h_2(t)={1\over 4\sqrt{\pi}}e^{-t^2}(2t^2-1)}, \qquad\displaystyle{h_3(t)= {1\over 12\sqrt{\pi}}e^{-t^2}t(2t^2-3)}.$$
Then we apply Theorem \ref{zeroshermite} and Corollary \ref{ruba} to get that
\begin{eqnarray*}
4\sqrt{\pi}\int_{-\infty}^\infty\vert h_2(t)\vert dt=\int_{-\infty}^\infty\vert 2t^2-1\vert e^{-t^2}dt&=&2\sqrt{2}e^{-1\over 2};\cr
4\sqrt{\pi}\int_{-\infty}^\infty\vert t\,h_2(t)\vert dt=\int_{-\infty}^\infty\vert t(2t^2-1)\vert e^{-t^2}dt&=&4e^{-1\over 2}-1;\cr
12\sqrt{\pi}\int_{-\infty}^\infty\vert h_3(t)\vert dt=\int_{-\infty}^\infty\vert t(2t^2-3)\vert e^{-t^2}dt&=&1+4e^{-3\over 2};\cr
12\sqrt{\pi}\int_{-\infty}^\infty\vert t\,h_3(t)\vert dt=\int_{-\infty}^\infty\vert t^2(2t^2-3)\vert e^{-t^2}dt&=&3\sqrt{6}e^{-3\over 2};\cr
12\sqrt{\pi}\int_{-\infty}^\infty\vert  t^2h_3(t)\vert dt=
\int_{-\infty}^\infty\vert t^3(2t^2-3)\vert e^{-t^2}dt&=&2\left(7e^{-3\over 2}-{1\over 2}\right).\cr
\end{eqnarray*}
}
\end{example}

\section{Jacobi polynomials}

\setcounter{theorem}{0}
\setcounter{equation}{0}

Jacobi polynomials $P_n^{(\alpha, \beta)}$ where
$$
P_n^{(\alpha, \beta)}(t):= \sum_{j=0}^n {1\over 2^n}{n+\alpha \choose n-j}{n+\beta\choose j}\left({t-1}\right)^{j}\left({t+1}\right)^{n-j}, \quad t\in \R,
$$ for $n\in \N$ and $\alpha, \beta \in \R$,
are polynomials solutions of second order differential equation \begin{equation}\label{ecdif}
(1-t^2)y''(t)+(\beta-\alpha-(\alpha+\beta+2)t)y'(t)+n(n+\alpha+\beta+1)y(t)=0. \end{equation}
Note that $P_n^{(\alpha, \beta)}(1)=\displaystyle{n+\alpha \choose n}$.

 Other interesting identities are the following ones,
\begin{eqnarray*}
P_n^{(\alpha, \beta)}(t)&=&(-1)^{n}P_n^{(\beta, \alpha)}(-t),\cr \cr
{d\over dt}P_n^{(\alpha, \beta)}(t)&=&{n+\alpha+\beta+1\over 2}P_{n-1}^{(\alpha+1, \beta+1)}(t), \quad t\in \R.
\end{eqnarray*}
First Jacobi polynomials are $P_0^{(\alpha, \beta)}(t)=1;$ and $P_1^{(\alpha, \beta)}(t)={1\over 2}(\alpha+\beta+2)t+{1\over 2}\left(\alpha-\beta\right)$. For $\alpha=\beta=0$, polynomials $P_n^{(0,0)}$ are the known as Legendre polynomials, see for example \cite[Chapter 4] {Szego}.

In the following we define Jacobi functions $p_n^{(\alpha,\beta)}$ by
\begin{eqnarray*}
p_n^{(\alpha,\beta)}(t):&=&\frac{(2n+\alpha+\beta+1)\Gamma(n+\alpha+\beta+1)n!}{2^{\alpha+\beta+1}\Gamma(n+\alpha+1)\Gamma(n+\beta+1)}(1-t)^{\alpha}(1+t)^{\beta}P_n^{(\alpha, \beta)}(t)\cr
&=&\frac{(-1)^n(2n+\alpha+\beta+1)\Gamma(n+\alpha+\beta+1)}{2^{n+\alpha+\beta+1}\Gamma(n+\alpha+1)\Gamma(n+\beta+1)}\frac{d^n}{dt^n}((1-t)^{n+\alpha}(1+t)^{n+\beta})(t),
\end{eqnarray*}
for $n\in \NN\cup\{0\}$ and $t\in (-1,1)$.  The following lemma contains some results for Jacobi functions $p_n^{(\alpha,\beta)}$ which are similar to some equalities for Jacobi polynomials $P_n^{(\alpha,\beta)}$.
\begin{lemma}\label{lemin} For $n\in \N$ and $\alpha, \beta \in \R$, the following equalities hold.
\begin{itemize}
\item[(i)]
$$
\left(p_n^{(\alpha,\beta)}\right)^{(k)}={(-1)^k\over 2^{k}}{\Gamma(n+\alpha+\beta+1)\over\Gamma(n+\alpha+\beta+1-k) }p_{n+k}^{(\alpha-k,\beta-k)}, \qquad k\in \N.
$$

\item[(ii)]
\begin{eqnarray*}
&\,&p_{n-1}^{(\alpha+1,\beta+1)}(t)={(n+\alpha+\beta)(n+\alpha+\beta+1)\over (2n+\alpha+\beta)(2n+\alpha+\beta-1)}p_{n-1}^{(\alpha,\beta)}(t) \cr&\,&+ {(n+\alpha+\beta+1)(\alpha-\beta)\over (2n+\alpha+\beta)(2n+\alpha+\beta+2)}p_{n}^{(\alpha,\beta)}(t)-{(n+\alpha+1)(n+\beta+1)\over (2n+\alpha+\beta+2)((2n+\alpha+\beta+3)}p_{n+1}^{(\alpha, \beta)}(t).
\end{eqnarray*}

\item[(iii)]

\begin{eqnarray*}
&\,&{(n+\alpha+1)(n+\beta+1)\over (2n+\alpha+\beta+3)}p_{n+1}^{(\alpha,\beta)}(t) = -{(2n+\alpha+\beta+2)(n+\alpha+\beta)n\over (2n+\alpha+\beta)((2n+\alpha+\beta-1)}p_{n-1}^{(\alpha, \beta)}(t)\cr
&\,&+{1\over 2(2n+\alpha+\beta)}\left({(2n+\alpha+\beta+2)(2n+\alpha+\beta)t+\alpha^2-\beta^2}\right)p_{n}^{(\alpha,\beta)}(t).
\end{eqnarray*}

\item[(iv)]

\begin{eqnarray*}
p_{n-1}^{(\alpha+1,\beta+1)}(t)={(n+\alpha+\beta)(2n+\alpha+\beta+1)\over (2n+\alpha+\beta)(2n+\alpha+\beta-1)}p_{n-1}^{(\alpha,\beta)}(t)+ {1\over 2}\left({\alpha-\beta\over 2n+\alpha+\beta}-t\right)p_{n}^{(\alpha,\beta)}(t).
\end{eqnarray*}

\end{itemize}

\end{lemma}

\begin{proof} To show the first part, note that
\begin{eqnarray*}
{d\over dt}p_n^{(\alpha,\beta)}(t)&=&\frac{(-1)^n(2n+\alpha+\beta+1)\Gamma(n+\alpha+\beta+1)}{2^{n+\alpha+\beta+1}\Gamma(n+\alpha+1)\Gamma(n+\beta+1)}\frac{d^{n+1}}{dt^{n+1}}((1-t)^{n+\alpha}(1+t)^{n+\beta})(t)\cr
&=&-{n+\alpha+\beta\over 2}p_{n+1}^{(\alpha-1,\beta-1)}(t).
\end{eqnarray*}
We iterative this equality to get $\left(p_n^{(\alpha,\beta)}\right)^{(k)}$.

The part (ii) is straightforward consequence of a formula similar to Jacobi polynomials, \cite[(4.5.5)]{Szego}. The part (iii) is obtained from the recurrence formula for Jacobi polynomials, \cite[(4.5.1)]{Szego}. To finish, the part (iv) is obtained from part (ii) and (iii).
\end{proof}

\begin{proposition} \label{yes} For $n, i\in \N$, $\alpha,\beta>-1$, the Jacobi functions $ p_n^{(\alpha,\beta)}$ verify the inequality
$$
\int_{-1}^1 {\vert  t\vert^i}\vert p_n^{(\alpha,\beta)}(t)\vert dt\le C_{\alpha,\beta}{\sqrt{n\over i^{1+\gamma}}},
$$where $C_{\alpha,\beta}$ is a independent constant of $n$ and $i$ and $\gamma=\min({\alpha, \beta)}$.

\end{proposition}

\begin{proof} We denote by $c_{n}^{(\alpha, \beta)}= \displaystyle{ \frac{(2n+\alpha+\beta+1)\Gamma(n+\alpha+\beta+1)n!}{2^{\alpha+\beta+1}\Gamma(n+\alpha+1)\Gamma(n+\beta+1)}}$. By the Cauchy-Schwarz inequality, we get that
\begin{eqnarray*}
&\,&\int_{-1}^1 {\vert  t\vert^i}\vert p_n^{(\alpha,\beta)}(t)\vert dt\le{{c_{n}^{(\alpha, \beta)}}}\left(\int_{-1}^1 t^{2i}(1-t)^\alpha(1+t)^\beta dt\right)^{1\over 2}\left(\int_{-1}^1 (1-t)^\alpha(1+t)^\beta\vert P_{n}^{(\alpha,\beta)}(t)\vert^2dt\right)^{1\over 2}\cr
&\quad&\qquad \qquad\le{\sqrt{c_{n}^{(\alpha, \beta)}}}\left(C_\alpha\int_{-1}^0t^{2i}(1+t)^\beta dt+C_\beta\int_{0}^1t^{2i}(1-t)^\alpha dt\right)^{1\over 2}\cr
&\quad&\qquad \qquad\le{C_{\alpha, \beta}\sqrt{c_{n}^{(\alpha, \beta)}}}\left({(2i)!\over\Gamma(2i+\beta+2)}+{(2i)!\over\Gamma(2i+\alpha+2)})\right)^{1\over 2}
\end{eqnarray*}
where we have used that functions $t\mapsto \sqrt{c_{n}^{(\alpha, \beta)}} (1-t)^{\alpha\over 2}(1+t)^{\beta\over 2}P_n^{(\alpha,\beta)}(t)$ form a Hilbertian basis on $L^2(-1,1)$.

Since $\displaystyle{\lim_{n\to \infty}\displaystyle{\Gamma(n+\alpha)\over (n-1)!\,n^\alpha}=1}$, we deduce that $\displaystyle{(2i)!\over\Gamma(2i+\beta+2)}\le C_\beta{1\over i^{\beta +1}}$, then

$$
c_{n}^{(\alpha, \beta)}
\le C_{\alpha, \beta}{n(n-1)!\,n^{\alpha+\beta+1}n!\over (n-1)!^2\,n^{\alpha+1}n^{\beta+1}}
\le  C_{\alpha,\beta} \,n$$
and we conclude the result.
\end{proof}

\begin{remark} {\rm In \cite[(7.34.1)]{Szego}, the equivalence $\int_{-1}^1\vert p_n^{(\alpha, \beta)}(t)\vert dt \sim \sqrt{n}$ when $n\to \infty$ is stated.}
\end{remark}
\medskip

The proof of next lemma is similar to the proof of Lemma \ref{laguerre} and we avoid it here.

\begin{lemma}
\label{jacobi} For $n\in \N$ and $0\le i\le n-1$, the Jacobi functions $p_n^{(\alpha, \beta)}$ verify
$$
\int  {t^i\over i!} p_n^{(\alpha, \beta)}(t) dt= -\sum_{k=0}^i{2^{k+1}{\Gamma(n+\alpha+\beta+1)}\over\Gamma(n+k+\alpha+\beta+2)} {t^{i-k}\over (i-k)!} p_{n-1-k}^{(\alpha+1+k, \beta+1+k)}(t).
$$
\end{lemma}

\begin{theorem}\label{jacobino} Take $n\in \N$, $0\le i\le n-1$; $\alpha, \beta>-1$ and  $\mathcal{Z}(P_{n}^{(\alpha, \beta)})=\{t_1<\ldots<t_{n_0} <t_{n_0+1}< t_{n}\}$  with   $0\in [t_{n_0},t_{n_0+1})$.

\begin{itemize}
\item[(i)] In the case that $i$ is even, the Jacobi functions $p_n^{(\alpha, \beta)}$ satisfy
 \begin{eqnarray*} \int_{-1}^1{\vert t\vert^i\over i!} \vert p_n^{(\alpha, \beta)}(t)\vert dt&=&\displaystyle\sum_{m=1}^{n}(-1)^{m+n}\displaystyle\displaystyle\sum_{k=0}^{i}{2^{k+2}{\Gamma(n+\alpha+\beta+1)}\over\Gamma(n+k+\alpha+\beta+2)} {t^{i-k}_m\over (i-k)!} p_{n-1-k}^{(\alpha+1+k, \beta+1+k)}(t_m).
\end{eqnarray*}
\item[(ii)] In the case that $i$ is odd, they verify
 \begin{eqnarray*} \int_{-1}^1{\vert t\vert^i\over i!}  \vert p_n^{(\alpha, \beta)}(t)\vert dt&=&\displaystyle\sum_{m=1}^{n_0}(-1)^{m+n+1}\displaystyle\displaystyle\sum_{k=0}^{i}{2^{k+2}{\Gamma(n+\alpha+\beta+1)}\over\Gamma(n+k+\alpha+\beta+2)} {t^{i-k}_m\over (i-k)!} p_{n-1-k}^{(\alpha+1+k, \beta+1+k)}(t_m)\cr
 &+&(-1)^{n_0+n}{2^{i+2}{\Gamma(n+\alpha+\beta+1)}\over\Gamma(n+i+\alpha+\beta+2)}p_{n-1-i}^{(\alpha+1+i, \beta+1+i)}(0)\cr&+ &\sum_{m={n_0+1}}^n(-1)^{m+n}\displaystyle\displaystyle\sum_{k=0}^{i}{2^{k+2}{\Gamma(n+\alpha+\beta+1)}\over\Gamma(n+k+\alpha+\beta+2)} {t^{i-k}_m\over (i-k)!} p_{n-1-k}^{(\alpha+1+k, \beta+1+k)}(t_m).
\end{eqnarray*}
 \end{itemize}

\end{theorem}

\begin{proof} (i) The proof is similar to the proof of Theorem \ref{zeroshermite} (i) due to $P_n^{(\alpha, \beta)}(1) >0$; in this case, we  apply Lemma \ref{jacobi}.

(ii) Suppose that $n$ is even. Again we denote by $t_0=-1$ and $t_{n+1}=1$. Then we show that
\begin{eqnarray*}&\quad&\int_{-1}^1{\vert t\vert^i\over i!} \vert p_n^{(\alpha,\beta)}(t)\vert dt= \frac{1}{i!}\sum_{m=0}^{{n_0-1}} (-1)^{m+1} \int_{t_m}^{t_{m+1}}t^ip_n^{(\alpha, \beta)}(t)dt+{(-1)^{n_0+1}\over i} \int_{t_{n_0}}^{0}t^ip_n^{(\alpha, \beta)}(t)dt\cr &\quad&\qquad+ {(-1)^{n_0}\over i!} \int^{t_{n_0+1}}_{0}t^ip_n^{(\alpha, \beta)}(t)(t)dt+ \frac{1}{i!}\sum^{n}_{m=n_0+1} (-1)^{m} \int_{t_m}^{t_{m+1}}t^ip_n^{(\alpha, \beta)}(t)(t)dt.\cr
\end{eqnarray*}

By Lemma \ref{jacobi}, we deduce that
\begin{eqnarray*}
\int^{t_{n_0+1}}_{0}{t^i\over i!}p_n^{(\alpha, \beta)}(t)dt&=&-\displaystyle\sum_{k=0}^{i}{2^{k+1}{\Gamma(n+\alpha+\beta+1)}\over\Gamma(n+k+\alpha+\beta+2)} {t_{n_0+1}^{i-k}\over (i-k)!}p_{n-1-k}^{(\alpha+1+k, \beta+1+k)}(t_{n_0+1})\cr
&+&{2^{i+1}{\Gamma(n+\alpha+\beta+1)}\over\Gamma(n+i+\alpha+\beta+2)} p_{n-1-i}^{(\alpha+1+i, \beta+1+i)}(0)
\end{eqnarray*}
and then
\begin{eqnarray*}
&\quad&{(-1)^{n_0}\over i!} \int^{t_{n_0+1}}_{0}t^ip_n^{(\alpha, \beta)}(t)(t)dt+ \frac{1}{i!}\sum^{n}_{m=n_0+1} (-1)^{m} \int_{t_m}^{t_{m+1}}t^ip_n^{(\alpha, \beta)}(t)(t)dt\cr
&\quad&={(-1)^{n_0}}{2^{i+1}{\Gamma(n+\alpha+\beta+1)}\over\Gamma(n+i+\alpha+\beta+2)} p_{n-1-i}^{(\alpha+1+i, \beta+1+i)}(0)\cr
&\quad&\qquad +  \sum_{m={n_0+1}}^n(-1)^{m}\displaystyle\displaystyle\sum_{k=0}^{i}{2^{k+2}{\Gamma(n+\alpha+\beta+1)}\over\Gamma(n+k+\alpha+\beta+2)} {t_{m}^{i-k}\over (i-k)!}p_{n-1-k}^{(\alpha+1+k, \beta+1+k)}(t_{m}).
\end{eqnarray*}
Analogously we get the identities for the first two summands and we conclude the result.

Finally the case when $n$ is odd is similar to the previous one.
\end{proof}

\begin{remark}{\rm In the case that $\alpha= \beta$, then $p_{2n}^{(\alpha,\alpha)}$ is a even function and $p_{2n-1}^{(\alpha,\alpha)}$ is odd. In this last case, $0\in \mathcal{Z}(P_{2n-1}^{(\alpha,\alpha)})$ and $n_0={n-1}$.}
\end{remark}

\begin{corollary}\label{calva} For $\alpha,\beta>-1$ and $n\in\N$, the Jacobi functions $p_n^{(\alpha, \beta)}$ verify $$ \lVert p_n^{(\alpha,\beta)}\rVert_1=\frac{4}{(n+\alpha+\beta+1)}\frac{(n+\alpha+\beta)}{(2n+\alpha+\beta)}
\frac{(2n+\alpha+\beta+1)}{(2n+\alpha+\beta-1)}\displaystyle\sum_{m=1}^n(-1)^{m+n}p_{n-1}^{(\alpha,\beta)}(t_m),
$$ with $t_j\in\mathcal{Z}(P_n^{(\alpha,\beta)})=\{t_1<\ldots<t_n\}.$
\end{corollary}
\begin{proof} By Theorem \ref{jacobino}, we get that
\begin{eqnarray*} \int_{-1}^1 \vert p_n^{(\alpha, \beta)}(t)\vert dt&=&\displaystyle\sum_{m=1}^{n}(-1)^{m+n}\displaystyle
{2^{2}{\Gamma(n+\alpha+\beta+1)}\over\Gamma(n+\alpha+\beta+2)}  p_{n-1}^{(\alpha+1, \beta+1)}(t_m)\cr
&=&\frac{4}{(n+\alpha+\beta+1)}\frac{(n+\alpha+\beta)}{(2n+\alpha+\beta)}
\frac{(2n+\alpha+\beta+1)}{(2n+\alpha+\beta-1)}\displaystyle\sum_{m=1}^n(-1)^{m+n}p_{n-1}^{(\alpha,\beta)}(t_m)
\end{eqnarray*}
where we have applied Lemma \ref{lemin} (iv).
\end{proof}

\begin{example} \label{jaejemplo}{\rm For $p_2^{(0,0)}(t)={5\over 4}(3t^2-1)$, we conclude that
$$
\int_{-1}^1{5\over 4} \vert 3t^2-1\vert dt= \frac{10\sqrt{3}}{9}; \qquad
\int_{-1}^1{5\over 4} \vert t(3t^2-1)\vert dt=\frac{25}{24}.
$$
In the case that $p_3^{(0,0)}(t)={7\over 4}(5t^3-3t)$, we get that
$$
\int_{-1}^1{7\over 4} \vert (5t^3-3t)\vert dt=\frac{91}{40}; \qquad
\int_{-1}^1{7\over 4} \vert t(5t^3-3t)\vert dt=\frac{42}{25}\sqrt{\frac{3}{5}}.
$$
Now we consider $p_2^{(1,0)}(t)={3\over 4}(-5t^3+3t^2+3t-1)$ and we obtain that
\begin{eqnarray*}
\int_{-1}^1{3\over 4}\vert-5t^3+3t^2+3t-1\vert dt&=&\frac{72}{125}\sqrt{6}; \cr
\int_{-1}^1{3\over 4}\vert t(-5t^3+3t^2+3t-1)\vert dt&=&\frac{18921}{25000}.
\end{eqnarray*}

}
\end{example}

\subsection*{Acknowledgements}Authors thank M. Alfaro, O. Ciaurri, F. Marcell\'{a}n, M. Rezola, L. Roncal and J.L. Varona some advices, comments and references to improve the final version of the paper.


\begin{thebibliography}{999}

\bibitem{AM1} L. Abadias, P. J. Miana: {\it $C_0$-semigroups and resolvent operators approximated by Laguerre expansions.}  ArXiv:1311.7542 (2013), 1--26.

\bibitem{AM2} L. Abadias, P. J. Miana: {\it Hermite expansions of $C_0$-groups and cosine functions.} ArXiv:1404.3871 (2014), 1--20.



\bibitem{BBMQ} B. Beckermann, J. Bustamante,
R. Mart\'{\i}nez-Cruz and J.M. Quesada: { \it Gaussian, Lobatto and Radau positive
quadrature rules with a prescribed abscissa,} Calcolo, 51,  (2014), 319-328.




\bibitem{BVM} M. C. de Bonis, B. della Vecchia, and G. Mastroianni: { \it Approximation of the Hilbert transform on the real line using Hermite zeros,} Math. of Computation,  71,  (2001), 1169-1188.

\bibitem{CS} G. Criscuolo and L. Scuderi: { \it Error bound for product quadrature rules in $L^1$-weighted norm,} Calcolo,  31,  (1994), 73-93.

\bibitem{DS} R.A. Devore and L. Ridgway Scott:{\it Error bounds for Gaussian quadrature and weighted-$L^1$ polynomial approximation,}  Siam J. Numer. Anal.,  21,  (1984), 400-412.





\bibitem{Gaut} W. Gautschi: { \it Orthogonal polynomials and quadrature,}  Elec. Trans. on Numer. Anal.  { 9,} (1999),  65--76.



\bibitem{GM} R. E. Greenwood and J.J. Miller: {\it Zeros of the Hermite polynomials and weights for Gauss' mechanical quadrature formula}, Bull. Amer. Math. Soc. {\bf 54}, Number 8 (1948), 765-769.



\bibitem{Hunter} D.B. Hunter: {\it Some Gauss-Type Formulae for the evaluation of Cauchy Principal Values of
Integrals}, Numer. Math. 19 (1972), 419-424.

\bibitem{Jacobi} C.G.J. Jacobi: {\it \"{U}ber Gau{\ss}s neue Methode, die Werthe der Integrale n\"{a}herungsweise zu finde}, J. Reine Angew. Math. 1 (1826), 301-308.


\bibitem{Lebedev} N. N. Lebedev: { \it Special functions and their applications,} Selected Russian Publications in the Mathematical Sciences. Prentice-Hall (1965).


\bibitem{Santos} J.C. Santos-Le\'{o}n: {\it Szeg\"{o} polynomials and Szeg\"{o} quadrature for the F\'{e}jer kernel}, J. of Comp. and Applied Math. 179 (2005) 327–341.

\bibitem{SZ} H. E. Salzer and R. Zucker: {\it Table of the zeros and weight factors of the first fifteen Laguerre polynomials},  Bull. Amer. Math. Soc. 55 (1949), 1004-1012.


\bibitem{Szego} G. Szeg\"{o}: {\it Orthogonal polynomials,} American Mathematical Society Colloquium Publications Volume XXIII. American Mathematical Society (1967).


\end{thebibliography}
\end{document}